\documentclass[a4paper]{amsart}

\usepackage[T1]{fontenc}
\usepackage{amsmath, amssymb, amsthm, mathrsfs, mathtools, graphicx}
\usepackage{varioref}
\usepackage[backref=page]{hyperref}
\usepackage{enumitem, xspace, ifthen, comment}
\usepackage[all]{xy}
\usepackage[nameinlink, capitalize]{cleveref}
\usepackage[dvipsnames]{xcolor}
\usepackage{tikz-cd}

\setlist[enumerate]{
  label=(\thethm.\arabic*),
  before={\setcounter{enumi}{\value{equation}}},
  after={\setcounter{equation}{\value{enumi}}},
  itemsep=1ex
}

\setlist[itemize]{
  leftmargin=*,
  topsep=1ex,
  itemsep=1ex,
  label=$\circ$
}

\newtheorem*{thm-plain}{Theorem}
\newtheorem{thm}{Theorem}[section]
\newtheorem{lem}[thm]{Lemma}
\newtheorem{prp}[thm]{Proposition}

\newtheorem{fct}[thm]{Fact}

\numberwithin{equation}{thm}

\theoremstyle{definition}
\newtheorem{dfn}[thm]{Definition}

\newtheorem*{dfn-plain}{Definition}
\newtheorem{cons}[thm]{Construction}

\theoremstyle{remark}
\newtheorem{clm}[thm]{Claim}

\newtheorem*{ntn-plain}{Notation}
\newtheorem{setup}[thm]{Setup}
\newtheorem{rem}[thm]{Remark}

\newtheorem{exm}[thm]{Example}
\newtheorem*{rem-plain}{Remark}

\newcommand{\inv}{^{-1}}
\newcommand{\from}{\colon}
\newcommand{\gdw}{\Leftrightarrow}
\newcommand{\imp}{\Rightarrow}
\newcommand{\lto}{\longrightarrow}

\newcommand{\x}{\times}

\newcommand{\surj}{\twoheadrightarrow}
\newcommand{\bij}{\xrightarrow{\,\smash{\raisebox{-.5ex}{\ensuremath{\scriptstyle\sim}}}\,}}
\newcommand{\isom}{\cong}
\newcommand{\defn}{\coloneqq}

\newcommand{\tensor}{\otimes}
\newcommand{\id}{\mathrm{id}}

\newcommand{\wt}{\widetilde}

\renewcommand{\d}{\mathrm d}

\newcommand{\dual}{^{\smash{\scalebox{.7}[1.4]{\rotatebox{90}{\textup\guilsinglleft}}}}}

\newcommand{\factor}[2]{\left. \raise 2pt\hbox{$#1$} \right/\hskip -2pt \raise -2pt\hbox{$#2$}}

\newdir{ ir}{{}*!/-5pt/@^{(}} 
\newdir{ il}{{}*!/-5pt/@_{(}} 

\DeclareMathOperator{\img}{im}

\DeclareMathOperator{\rk}{rk}
\DeclareMathOperator{\Sym}{Sym}
\DeclareMathOperator{\Aut}{Aut}
\DeclareMathOperator{\Out}{Out}
\DeclareMathOperator{\Inn}{Inn}

\newcommand{\lref}{\labelcref}

\newcommand{\set}[1]{\left\{ #1 \right\}}

\def\rd#1.{\lfloor{#1}\rfloor}
\def\rp#1.{\lceil{#1}\rceil}
\def\tw#1.{\langle{#1}\rangle}

\renewcommand{\O}[1]{\mathscr{O}_{#1}}
\newcommand{\Omegap}[2]{\Omega_{#1}^{#2}}

\newcommand{\T}[1]{\mathcal{T}_{#1}}

\newcommand{\reg}[1]{{#1}_{\mathrm{reg}}}

\def\Hnought#1.#2.{\mathit{\Gamma} \!\left( #1, #2 \right)}
\def\HH#1.#2.#3.{\mathrm{H}^{#1} \!\left( #2, #3 \right)}
\def\hh#1.#2.#3.{h^{#1} \!\left( #2, #3 \right)}
\def\RR#1.#2.#3.{R^{#1} #2_* #3}
\def\HHc#1.#2.#3.{\mathrm{H}_{\mathrm{c}}^{#1} \!\left( #2, #3 \right)}
\def\Hh#1.#2.#3.{\mathrm{H}_{#1} \!\left( #2, #3 \right)}
\def\Hom#1.#2.{\mathrm{Hom} \!\left( #1, #2 \right)}
\def\End#1.{\mathrm{End} \!\left( #1 \right)}
\def\sHom#1.#2.{\mathscr{H}\!om \!\left( #1, #2 \right)}
\def\Ext#1.#2.#3.{\mathrm{Ext}^{#1} \!\left( #2, #3 \right)}
\def\sExt#1.#2.#3.{\mathscr{E}\!xt^{#1} \!\left( #2, #3 \right)}
\def\Link#1.#2.{\mathrm{Link} \!\left( #1, #2 \right)}

\newcommand{\GL}[2]{\mathrm{GL}(#1, #2)}

\newcommand{\SO}[1]{\mathrm{SO}(#1)}
\newcommand{\U}[1]{\mathrm{U}(#1)}
\newcommand{\SU}[1]{\mathrm{SU}(#1)}
\newcommand{\ke}{\mathrm{KE}}

\newcommand{\kahler}{K{\"{a}}hler\xspace}

\newcommand{\qe}{quasi-\'etale\xspace}
\newcommand{\lz}{Lipman--Zariski Conjecture\xspace}

\DeclareMathOperator{\vol}{vol}

\DeclareMathOperator{\tr}{tr}

\renewcommand{\theta}{\vartheta}
\renewcommand{\phi}{\varphi}

\newcommand{\Z}{\ensuremath{\mathbb Z}}

\newcommand{\R}{\ensuremath{\mathbb R}}
\newcommand{\C}{\ensuremath{\mathbb C}}
\newcommand{\bB}{\ensuremath{\mathbb B}}

\newcommand{\bP}{\ensuremath{\mathbb P}}

\newcommand{\bT}{\ensuremath{\mathbb T}}

\newcommand{\frg}{\mathfrak g}  
 \newcommand{\frk}{\mathfrak k}

  \newcommand{\sC}{\mathscr C}
 \newcommand{\sE}{\mathscr E}

\newcommand{\cD}{\mathcal D} \newcommand{\cE}{\mathcal E}

\newcommand{\cS}{\mathcal S}

\definecolor{forrest}{RGB}{81,133,49}
\definecolor{mydarkblue}{RGB}{10,92,153}

\title[Slope zero tensors, uVHS and quotients of tube domains]{Slope zero tensors, uniformizing variations of Hodge structure and quotients of tube domains}

\author{Patrick Graf}
\address{Lehrstuhl f\"ur Mathematik I, Universit\"at Bayreuth, 95440 Bayreuth, Germany}
\email{\href{mailto:patrick.graf@uni-bayreuth.de}{patrick.graf@uni-bayreuth.de}}
\urladdr{\href{https://patrickgraf.gitlab.io/en/}{https://patrickgraf.gitlab.io/en/}}

\author{Aryaman Patel}
\address{AG Lazi\'c, Universit\"at des Saarlandes, 66123 Saarbr\"ucken, Germany}
\email{\href{mailto:aryaman.patel@math.uni-sb.de}{aryaman.patel@math.uni-sb.de}}
\urladdr{\href{https://sites.google.com/view/aryamanpatel/home}{https://sites.google.com/view/aryamanpatel/}}

\date{October 16, 2025}

\hypersetup{
  pdfauthor={Patrick Graf and Aryaman Patel},
  pdftitle={Slope zero tensors, uniformizing variations of Hodge structure and quotients of tube domains},
  pdfkeywords={Varieties with ample canonical divisor, uniformization, klt singularities, Kähler--Einstein metrics, holonomy groups, tube domains},
  pdfstartview={Fit},
  pdfpagelayout={OneColumn},
  pdfpagemode={UseNone},
  linktoc=all,
  breaklinks,
  linkcolor=[RGB]{0 0 96},
  citecolor=[RGB]{96 0 0},
  urlcolor=[RGB]{0 96 0},
  colorlinks}

\begin{document}

\begin{abstract}
We prove an equivalence between two approaches to characterizing complex-projective varieties $X$ with klt singularities and ample canonical divisor that are uniformized by bounded symmetric domains.
In order to do so, we show how to construct a uniformizing variation of Hodge structure from a slope zero tensor and vice versa.
As a consequence, we generalize various uniformization results of Catanese and Di Scala to the singular setting.
For example, we prove that $X$ is a quotient of a bounded symmetric domain of tube type by a group acting properly discontinuously and freely in codimension one if and only if $X$ admits a slope zero tensor.
As a key step in the proof, we establish the compactness of the holonomy group of the singular K\"{a}hler--Einstein metric on $X_{\mathrm{reg}}$.
\end{abstract}

\maketitle

\tableofcontents

\section{Introduction}

The problem of characterizing smooth projective varieties $X$ uniformized by a bounded symmetric domain $\cD$ is classical and has been thoroughly studied by Yau, Simpson and Catanese--Di Scala, among others~\cite{Yau78, Simpson88, CataneseDiScala13, CataneseDiScala14}.
It has also been solved more generally for projective varieties with klt singularities and for orbifolds~\cite{Patel23, GrafPatel24, GrafPatel25, Catanese25}.

There are two different approaches to the problem:
\begin{itemize}
\item Simpson's approach is of a Hodge-theoretic nature and provides a criterion for uniformization by $\cD = \factor{G_0}{K_0}$ in terms of the existence of a \emph{uniformizing variation of Hodge structure (uVHS)} for $G_0$ on $X$.
The existence of such a uVHS is in turn equivalent to a reduction of structure group of the tangent bundle $\T X$ to $K = (K_0)_\C$ together with a certain Chern class equality depending on $\cD$.
\item The approach of Catanese--Di Scala is more differential-geometric and provides a uniformization criterion in terms of the existence of a holomorphic tensor $\sigma$ on~$X$ having certain properties.
This yields several uniformization results, the most general of which (in terms of $\cD$) is as follows: $X$ is uniformized by a bounded symmetric domain $\cD$ without higher-dimensional ball factors if and only if there is a tensor $\sigma$ with the three properties listed in \cref{main4}.
\end{itemize}

Now fix a bounded symmetric domain $\cD = \factor{G_0}{K_0}$ without higher-dimensional unit balls as factors.
It is then clear that on $X$, the existence of a uVHS for $G_0$ is equivalent to the existence of a tensor $\sigma$ as above (since both are equivalent to the universal cover of $X$ being isomorphic to $\cD$).
However, to the best of our knowledge there is yet no way to see this equivalence directly, i.e.~without passing to the universal cover.
One of the main goals of this article is to build a bridge between these two very different objects, in the more general setting of projective varieties with klt singularities.
More precisely, we construct a uVHS given a tensor (in the sense of Catanese--Di Scala) and vice versa, without passing to the universal cover.
As a consequence, we can generalize the uniformization results of Catanese--Di Scala to the singular setting.

\begin{thm}[cf.~\protect{\cite[Thm.~1.2]{CataneseDiScala13}}] \label{main2}
Let $X$ be an $n$-dimensional normal projective variety with klt singularities and such that $K_X$ is ample.
The following are equivalent:
\begin{enumerate}
\item\label{main2.1} We have $X \isom \factor \cD \Gamma$, where $\cD = \cD_1 \x \cdots \x \cD_m$ is a product of bounded symmetric domains of tube type such that $\rk \cD_i$ divides $\dim \cD_i$ for each $i = 1, \dots, m$, and $\Gamma \subset \Aut(\cD)$ is a discrete cocompact subgroup whose action is free in codimension one.
\item\label{main2.3} There is a semispecial tensor $\psi$ on $X$ (cf.~\cref{def semispecial}).
\end{enumerate}
\end{thm}

\begin{thm}[cf.~\protect{\cite[Thm.~1.3]{CataneseDiScala13}}] \label{main3}
Let $X$ be an $n$-dimensional normal projective variety with klt singularities and such that $K_X$ is ample.
The following are equivalent:
\begin{enumerate}
\item\label{main3.1} We have $X \isom \factor \cD \Gamma$, where $\cD$ is a bounded symmetric domain of tube type and $\Gamma \subset \Aut(\cD)$ is a discrete cocompact subgroup whose action is free in codimension one.
\item\label{main3.3} There is a slope zero tensor $\psi$ on $X$ (cf.~\cref{def slope zero}).
\end{enumerate}
\end{thm}

\begin{thm}[cf.~\protect{\cite[Thm.~1.2]{CataneseDiScala14}}] \label{main4}
Let $X$ be an $n$-dimensional normal projective variety with klt singularities and such that $K_X$ is ample.
The following are equivalent:
\begin{enumerate}
\item\label{main4.1} We have $X \isom \factor \cD \Gamma$, where $\cD$ is a bounded symmetric domain none of whose irreducible factors is isomorphic to a unit ball $\bB^m$ of dimension $m \ge 2$ and $\Gamma \subset \Aut(\cD)$ is a discrete cocompact subgroup whose action is free in codimension one.
\item\label{main4.3} There is a tensor
\[ \sigma \in \HH0.X.\sE nd \big( \T X [\tensor] \Omegap X1 \big). = \HH0.\reg X.\sE nd \big( \T{\reg X} \tensor \Omegap{\reg X}1 \big). \]
such that:
\begin{enumerate}[label=(\thethm.\arabic*),leftmargin=*]
\setcounter{enumii}{2}
\item\label{p1} There is a point $p \in \reg X$ and a splitting of the tangent space $T \defn T_p X$ as
\[ T = T_1' \oplus \cdots \oplus T_m' \]
such that the \emph{first Mok characteristic cone} $\mathcal{CS}$ of $\sigma_p$ is not all of $T$ and $\mathcal{CS}$ splits into $m$ irreducible components $\mathcal{CS}'(j)$ with
\item\label{p2} $\mathcal{CS}'(j) = T_1' \x \cdots \x \mathcal{CS}_j' \x \cdots \times T_m'$, where
\item\label{p3} $\mathcal{CS}_j' \subset T_j'$ is the cone over a non-degenerate (i.e.~$\mathcal{CS}_j'$ spans $T_j'$) smooth projective variety $\cS_j'$, unless $\mathcal{CS}_j' = \set{ 0 }$ and $\dim T_j' = 1$.
\end{enumerate}
\end{enumerate}
\end{thm}

\noindent
See \cref{def mok} for the definition of the first Mok characteristic cone.

We remark that although these results are formulated in terms of uniformization (i.e.~without reference to the existence of a uVHS), the proofs are actually in the form promised above.
The uniformization statements are then deduced by using \cref{prp uVHS} below.
For example, instead of \cref{main4} we actually prove the following.

\begin{thm} \label{main5}
Let $X$ be an $n$-dimensional normal projective variety with klt singularities and such that $K_X$ is ample.
The following are equivalent:
\begin{enumerate}
\item\label{main5.1} The smooth locus $\reg X$ admits a uniformizing VHS for a Hodge group $G_0$ of Hermitian type such that $\frg_0$ has no factors isomorphic to $\mathfrak{su}(p, 1)$ with $p \ge 2$.
\item\label{main5.2} There is a holomorphic tensor $\sigma \in \HH0.X.\sE nd \big( \T X [\tensor] \Omegap X1 \big).$ satisfying the three conditions~\lref{p1}--\lref{p3}.
\end{enumerate}
\end{thm}

\subsection*{Compactness of the holonomy group}

As a key technical step in the proof, we establish the following result, which might be of independent interest.
Recall that for a projective variety $X$ with klt singularities and $K_X$ ample, the smooth locus $\reg X$ carries a \kahler--Einstein metric $\omega_\ke$~\cite{EGZ09, BermanGuenancia14}.
It is relatively easy to deduce information about the restricted holonomy group $H^\circ$ of this metric from the existence of a tensor as in the above theorems, using the Bochner principle.
Computing the full holonomy group $H$, however, is difficult, as it might even have infinitely many components.
\cref{hol cover intro} overcomes this problem by showing that $H$ is in fact compact.
Since the tangent bundle $\T{\reg X}$ only admits a reduction of structure group to $H$, but not to $H^\circ$, this statement is crucial in order to pass from the given $X$ to the appropriate bounded symmetric domain $\cD$ (roughly speaking, by considering $\cD = \factor{G_0}{K_0}$ with $K_0 = H$).

\begin{thm}[Holonomy cover] \label{hol cover intro}
Let $X$ be an $n$-dimensional normal projective variety with klt singularities and such that $K_X$ is ample.
\begin{enumerate}
\item The holonomy group $H$ of $(\reg X, \omega_\ke)$ is compact.
\item There is a finite \qe Galois cover $\gamma \from Y \to X$ such that the holonomy group of $(\reg Y, \omega_{\ke, Y})$ is connected.
\end{enumerate}
\end{thm}

\noindent
In the Ricci-flat case, analogous statements have been proved in~\cite[Thm.~B]{GGK19} and~\cite[Thm.~C]{BochnerCGGN}.

\begin{rem-plain}
It is claimed in~\cite[Prop.~10.114]{Besse87} that for symmetric spaces, the holonomy representation always has finite index inside its normalizer.
This would easily imply \cref{hol cover intro}, however~\cite{Besse87} contains no proof and the reference given there~\cite[Cor.~7.2]{Wolf62} only proves the weaker statement that the holonomy of a \emph{compact} locally symmetric Riemannian manifold is compact.
Unless $X$ is smooth, $\reg X$ is clearly not compact and not even complete~\cite[Prop.~4.2]{GGK19}.
Therefore we cannot apply~\cite{Wolf62} in our situation.
In fact, since our proof only works for Hermitian symmetric spaces, we do not know whether~\cite[Prop.~10.114]{Besse87} is valid for general Riemannian symmetric spaces (but we do not need this level of generality).
\end{rem-plain}

\subsection*{Sketch of proof}

The proof of \cref{main5} is divided into five steps:
\begin{enumerate}[label=(\theenumi), leftmargin=*]
\setcounter{enumi}{0}
\item We consider the singular \kahler--Einstein metric $\omega_\ke$ on $X$.
By passing to the holonomy cover (\cref{hol cover intro}), we may assume that the holonomy group~$H$ of $(\reg X, \omega_\ke)$ is connected.
\item We use the Bochner principle to show that each irreducible factor $H_i$ of $H$ is the holonomy of a bounded symmetric domain $\cD_i$ not isomorphic to a unit ball $\bB^m$ with $m \ge 2$.
\item The tangent bundle $\T{\reg X}$ admits a reduction of structure group to $H$.
In other words, the frame bundle of $\reg X$ contains a principal $H$-bundle $P_H$.
Letting $H_\C$ denote the complexification of $H$, the principal $H_\C$-bundle $P \defn P_H \x_H H_\C$ provides a uniformizing system of Hodge bundles $(P, \theta)$ on $\reg X$.
\item The holonomy bundle $P_H$ is a metric on this system of Hodge bundles (in the sense of \cref{def metric}) and makes it into a uniformizing VHS.
\item The final step is to prove the converse, i.e.~that we obtain a tensor satisfying \lref{p1}--\lref{p3} starting with a uniformizing VHS.
\end{enumerate}

\subsection*{Acknowledgements}

We would like to thank Henri Guenancia for helpful discussions.
The first author was funded by the Deutsche Forschungsgemeinschaft (DFG, German Research Foundation) -- Projektnummer 521356266.
The second author was supported by the Deutsche Forschungsgemeinschaft (DFG, German Research Foundation) -- Project ID 286237555 (TRR 195) and Project ID 530132094.

\section{Preliminaries}

\subsection{Hodge groups and Hermitian symmetric spaces}

We recall some basic facts about Hermitian symmetric spaces.
This material can be found in~\cite{Simpson88} and~\cite{CMP17}.

\begin{dfn}[Hodge group of Hermitian type]
A \emph{Hodge group of Hermitian type} is a semisimple real algebraic group without compact factors $G_0$ whose complexified Lie algebra $\frg \defn \frg_0 \tensor_\R \C$ carries a Hodge decomposition
\[ \frg = \frg^{-1,1} \oplus \frg^{0,0} \oplus \frg^{1,-1} \]
such that the Lie bracket of $\frg$ is compatible with the Hodge decomposition in the sense that $[ \frg^{p,-p}, \frg^{q,-q} ] \subset \frg^{p+q,-p-q}$ for all $p, q \in \set{ -1, 0, 1 }$.
Furthermore the bilinear form on $\frg$ given by $(-1)^{p + 1} \tr( \operatorname{ad}(U) \circ \operatorname{ad}(\overline V) )$ on $\frg^{p,-p}$ must be positive definite.
\end{dfn}

Let $K_0$ denote the real subgroup of $G_0$ corresponding to the Lie algebra $\frg_0^{0,0} = \frg^{0,0} \cap \frg_0$.
Then $K_0$ is the largest subgroup such that the adjoint action of $K_0$ on $\frg$ preserves the Hodge decomposition of $\frg$.

We note the following classical facts, which will be needed later.

\begin{fct}
Let $G_0$ and $K_0$ be as above.
\begin{enumerate}
\item $K_0$ is a maximal compact subgroup of $G_0$.
\item The quotient $\cD = \factor{G_0}{K_0}$ is a Hermitian symmetric space of non-compact type (= bounded symmetric domain).
Moreover, every bounded symmetric domain can be expressed as such a quotient, by taking $G_0 = \Aut(\cD)$.
\item Given $\cD$, the groups $G_0$ and $K_0$ are uniquely determined up to isogeny.
In particular, the Lie algebra $\frg$ is determined by $\cD$.
\end{enumerate}
\end{fct}

Let $X$ be a smooth quasi-projective variety or a complex manifold, and let $G_0, K_0$ be as above.
We denote by $G$ and $K$ the complexifications of the groups $G_0$ and $K_0$, respectively.

\begin{dfn}[Uniformizing system of Hodge bundles]
A \emph{uniformizing system of Hodge bundles for $G_0$} on $X$ is a pair $(P, \theta)$, where $P$ is a holomorphic principal $K$-bundle on $X$ and $\theta$ is an isomorphism of vector bundles
\begin{align} \label{ushb}
\theta \from \T X \bij P \x_K \frg^{-1,1}
\end{align}
such that $[\theta(u), \theta(v)] = 0$ for all local sections $u, v$ of $\T X$.
\end{dfn}

If $(P, \theta)$ is a uniformizing system of Hodge bundles on $X$, then we also have an isomorphism $\Omegap X1 \isom P \x_K \frg^{1,-1}$.
Setting $\cE^{i,-i} \defn P \x_K \frg^{i,-i}$, we can form the direct sum
\[ \cE \defn P \x_K \frg = \cE^{-1,1} \oplus \cE^{0,0} \oplus \cE^{1,-1}. \]
There is a natural Higgs field $\cE \to \cE \tensor \Omegap X1$, which we also denote by $\theta$, given by $u \mapsto ( v \mapsto [ \theta(v), u ] )$.
It is easily checked that $\theta(\cE^{i,-i}) \subset \cE^{i-1,-i+1} \tensor \Omegap X1$ and that $\theta \wedge \theta = 0$ (due to the Jacobi identity in $\frg$ and the fact that $[\theta(u), \theta(v)] = 0$ for all local vector fields $u, v$).
So $\cE$ is an honest system of Hodge bundles, which justifies the name.

\begin{dfn}[Metric] \label{def metric}
Let $(P, \theta)$ be as above.
A \emph{metric} $h$ on $P$ is a $\sC^\infty$-reduction in structure group of $P$ from $K$ to $K_0$, i.e.~a principal $K_0$-subbundle $P_H \subset P$.
We then automatically have $P \isom P_H \x_{K_0} K$.
\end{dfn}

Recall the following well-known facts from the theory of principal bundles.

\begin{prp} \label{conn principal bundle}
Let $M$ be a smooth manifold, $G$ a Lie group with Lie algebra $\frg$, and $f \from P \to M$ a $G$-principal bundle.
\begin{enumerate}
\item To give a connection on $P$ (in the sense of a $G$-invariant horizontal subbundle of $\T P$) is equivalent to giving a $\frg$-valued $1$-form $\omega$ on $P$ which is $G$-equivariant and whose restriction to the fibres of $f$ equals the Maurer--Cartan form of $G$.
\item If $\omega_1, \omega_2$ are two connections on $P$, then their difference $\omega_1 - \omega_2$ is a $\frg$-valued $G$-equivariant $1$-form on $P$ which is furthermore horizontal (i.e.~it vanishes on the vertical subbundle of $\T P$).
\item To give a $\frg$-valued $G$-equivariant horizontal $1$-form on $P$ is equivalent to giving a map of vector bundles
\[ \T M \lto P \x_G \frg. \]
Hence the set of all connections on $P$ is an affine space over the vector space of $(P \x_G \frg)$-valued $1$-forms on $M$. \qed
\end{enumerate}
\end{prp}

We now revert to the notation used before \cref{conn principal bundle}.
We borrow the following construction from~\cite[Section~8]{Simpson88}, which is necessary to define a uniformizing variation of Hodge structure.

\begin{cons} \label{flat metric}
Let $(P, \theta)$ be a uniformizing system of Hodge bundles on $X$ equipped with a metric $P_H \subset P$, which we denote by $h$.
If $V$ is a polarized Hodge representation of $G_0$ in the sense of~\cite[p.~900]{Simpson88}, then $h$ induces a Hermitian metric (in the usual sense) on the associated vector bundle $P \x_K V$.
In particular, this applies to $V = \frg$.
By abuse of notation, we denote again by $h$ the Hermitian metric on $\cE = P \x_K \frg = P_H \x_{K_0} \frg$.

There is a unique connection $d_h$ on $P_H$ which is compatible with the holomorphic structure of $P$~\cite{AAB00}.
We can push forward $d_h$ to a connection on $R_H \defn P_H \x_{K_0} G_0$, which we still denote by $d_h$.
Let $\overline\theta \from \T X \to P \x_K \frg^{1,-1}$ be the conjugate of $\theta$.
Then $\theta + \overline\theta \from \T X \to P \x_K \frg_0$ and hence by \cref{conn principal bundle}, $D_h \defn d_h + \theta + \overline\theta$ is a connection on $R_H$.
\end{cons}

\begin{dfn}[Uniformizing VHS]
A \emph{uniformizing variation of Hodge structure (uVHS) for $G_0$} is a uniformizing system of Hodge bundles $(P, \theta)$ for $G_0$ together with a metric $h$ such that the connection $D_h$ is flat.
\end{dfn}

If $(P, \theta)$ is a uVHS, then $D_h$ induces a flat connection on $\sE = P \x_K \frg = R_H \x_{G_0} \frg$.
This connection can also be constructed as follows: let $d_h'$ be the Chern connection on $(\sE, h)$, and let $\theta$ be the Higgs field on $\sE$ as explained above.
Let $\theta^h$ be the adjoint of $\theta$ with respect to $h$, i.e.~we have
\[ \langle \theta u, v \rangle_h = \langle u, \theta^h v \rangle_h \]
for all local sections $u, v$ of $\sE$.
Then $D_h' \defn d_h' + \theta + \theta^h$ is the flat connection induced by $D_h$.

\begin{exm} \label{bsd uvhs}
Let $\cD = \factor{G_0}{K_0}$ be a bounded symmetric domain.
Then $G_0 \to \cD$ is a principal $K_0$-bundle, which we denote by $P_{\cD, H}$.
Set $P_\cD \defn P_{\cD, H} \x_{K_0} K$.
Then there is an isomorphism $\theta_\cD \from \T \cD \bij P_\cD \x_K \frg^{-1,1}$, where $K$ acts on $\frg^{-1,1} \subset \frg$ via the adjoint representation.
Therefore $(P_\cD, \theta_\cD)$ is a uniformizing system of Hodge bundles on $\cD$.
The subbundle $P_{\cD, H}$ is a metric $h$ on $(P_\cD, \theta_\cD)$.
The induced connection $D_h$ on $R_H \defn P_{\cD, H} \x_{K_0} G_0 \isom \cD \x G_0$ is given by pulling back the Maurer--Cartan form on $G_0$.
In particular, $D_h$ is flat.
This means that $(P_\cD, \theta_\cD)$ together with the metric $P_{\cD, H}$ is a uniformizing VHS on $\cD$.
\end{exm}

The following proposition has already been mentioned in the introduction.

\begin{prp}[cf.~\protect{\cite[Prop.~9.1]{Simpson88}}] \label{prp uVHS}
Let $X$ be an $n$-dimensional normal projective variety with klt singularities and such that $K_X$ is ample.
Let $\cD$ be a bounded symmetric domain.
The following are equivalent:
\begin{enumerate}
\item\label{uVHS.1} We have $X \isom \factor \cD \Gamma$, where $\Gamma \subset \Aut(\cD)$ is a discrete cocompact subgroup whose action is free in codimension one.
\item\label{uVHS.2} $\reg X$ admits a uniformizing VHS $(P, \theta)$ for some Hodge group $G_0$ of Hermitian type with $\cD \isom \factor{G_0}{K_0}$.
\end{enumerate}
\end{prp}

\begin{proof}
``\lref{uVHS.1} $\imp$ \lref{uVHS.2}'': Let $G_0 \defn \Aut(\cD)$, then $G_0$ is a Hodge group of Hermitian type and $\cD \isom \factor{G_0}{K_0}$.
Consider the uniformizing VHS $(P_\cD, \theta_\cD)$ on $\cD$ given as in \cref{bsd uvhs}.
Then in particular we have an isomorphism $\T \cD \isom G_0 \x_{K_0} \frg^{-1,1}$.

Let $\pi \from \cD \to X$ be the quotient map, and set $\cD^\circ \defn \pi\inv(\reg X)$.
The complement of $\cD^\circ$ in $\cD$ has complex codimension at least two.
Therefore the restriction $\pi\big|_{\cD^\circ} \from \cD^\circ \to \reg X$ is the universal cover of $\reg X$, and $\reg X \isom \factor{\cD^\circ}\Gamma$.
In particular, note that $\pi_1(\reg X) = \Gamma$.

Restrict $(P_\cD, \theta_\cD)$ to a uniformizing VHS $(P_{\cD^\circ}, \theta_{\cD^\circ})$ on $\cD^\circ$.
Consider the action of $\Gamma$ on $G_0$ given by multiplication on the left.
This stabilizes $G_0\big|_{\cD^\circ}$, hence the principal bundle $G_0\big|_{\cD^\circ}$ is $\Gamma$-equivariant.
Therefore it descends to $\reg X$.
That is, there is a principal $K_0$-bundle~$P_0$ on $\reg X$ such that $\pi^* P_0 \isom G_0\big|_{\cD^\circ}$ as principal $K_0$-bundles on $\cD^\circ$.
It follows from the proof of~\cite[Prop.~9.1]{Simpson88} that the isomorphism $\theta_{\cD^\circ}$ is also $\Gamma$-equivariant.
Therefore on $\reg X$ there is an isomorphism $\theta \from \T{\reg X} \bij P_0 \x_{K_0} \frg^{-1,1}$.
Setting $P \defn P_0 \x_{K_0} K$, we have $\theta \from \T{\reg X} \bij P \x_K \frg^{-1,1}$.
Hence $(P, \theta)$ is a uniformizing system of Hodge bundles on $\reg X$.

We must show that the metric $P_0 \subset P$ makes $(P, \theta)$ into a uniformizing VHS.
The connection $D_h$ on $R \defn P_0 \x_{K_0} G_0$ from \cref{flat metric} pulls back under the map $\pi$ to a connection $D_{\cD^\circ}$ on the trivial bundle $R_{\cD^\circ} \defn G_0\big|_{\cD^\circ} \x_{K_0} G_0$ on $\cD^\circ$, and is constructed in the same way.
Since $D_{\cD^\circ}$ is flat, it follows that $D_h$ is flat.
Thus $(P, \theta)$ is a uniformizing VHS on $\reg X$, as desired.

``\lref{uVHS.2} $\imp$ \lref{uVHS.1}'': By the definition of a uniformizing VHS, $P$ is a principal $K$-bundle, $\theta$ is an isomorphism $\T{\reg X} \bij P \x_K \frg^{-1,1}$ and there is a principal $K_0$-bundle $P_H \subset P$ and a flat connection on the principal $G_0$-bundle $R_H \defn P_H \x_{K_0} G_0$.
This means that there is also a flat connection on the associated vector bundle $\sE \defn R_H \x_{G_0} \frg = P \x_K \frg$ on $\reg X$, which contains $\T{\reg X}$ as a direct summand.

By~\cite[Thm.~1.14]{GKP16}, there is a finite \qe Galois cover $\gamma \from Y \to X$ such that $\gamma^* \sE$ extends to a flat locally free sheaf $\sE_Y$ on all of $Y$, which then contains $\T Y$ as a direct summand.
Therefore $\T Y$ is locally free, and $Y$ is smooth by the \lz for klt spaces~\cite{GK13}.

We may now apply~\cite[Prop.~9.1]{Simpson88} and conclude that the universal cover $\wt Y$ of $Y$ is isomorphic to $\cD$.
This implies~\lref{uVHS.1} by standard arguments involving Selberg's lemma, cf.~the proof of~\cite[Thm.~1.3]{GKPT1_ENS}.
\end{proof}

The following notions appear in the main results.
They are taken from~\cite{CataneseDiScala13} and~\cite{CataneseDiScala14}.
Let $X$ be an $n$-dimensional normal projective variety with klt singularities and such that $K_X$ is ample.

\begin{dfn} \label{def semispecial}
A \emph{semispecial tensor} on $X$ is a nonzero section
\[ 0 \ne \psi \in \HH0.X.\Sym^{n}(\Omegap X1)(-K_X) [\tensor] \eta. \]
for some rank one reflexive sheaf $\eta$ on $X$ such that $\eta^{[2]} \isom \O X$.
\end{dfn}

\begin{dfn} \label{def slope zero}
A \emph{slope zero tensor} on $X$ is a nonzero section
\[ 0 \ne \psi \in \HH0.X.\Sym^{mn}(\Omegap X1)(-mK_X). \]
for some positive integer $m > 0$.
\end{dfn}

\begin{dfn}[First Mok characteristic cone] \label{def mok}
Let $T$ be a finite-dimensional complex vector space and $\sigma \in \End T \tensor T\dual.$.
The \emph{first Mok characteristic cone} $\mathcal{CS}$ of $\sigma$ is defined as
\[ \set{ t \in T \;\;\big|\;\; \exists \alpha \in T\dual \setminus \set0 \!: t \tensor \alpha \in \ker(\sigma) }. \]
\end{dfn}

\subsection{Classification of irreducible Hermitian symmetric spaces}

For later reference, \cref{ihss classification} summarizes the classification of irreducible Hermitian symmetric spaces and their most important invariants.
This information can be found in~\cite{Helgason78} or~\cite{Besse87} and in~\cite[p.~525]{FKKLR00}.

\begin{rem-plain}
In the first column, we have used the original numbering scheme by Cartan.
Modern references such as~\cite{CataneseDiScala13} and~\cite{FKKLR00} switch types III and IV.
\end{rem-plain}

\begin{dfn}[Domains of tube type]
A bounded symmetric domain $\cD$ is said to be \emph{of tube type} if it is biholomorphic to a \emph{tube domain} $\Omega = V \oplus \mathrm i \sC$, where $V$ is a real vector space and $\sC \subset V$ is an open self-dual convex cone containing no lines.
\end{dfn}

\renewcommand{\arraystretch}{2}

\begin{table}[t]
\centering
\begin{tabular}{|c|c|c|c|c|c|}
\hline
\cite{Cartan35} & \cite{Helgason78} & \textbf{Symmetric space} & \textbf{Dimension} & \textbf{Rank} & \textbf{Tube type} \\
\hline \hline
$\mathrm I_{p, q}$ & A~III & $\factor{\SU{p, q}}{\mathrm S \big( \U p \x \U q \big)}$ & $pq$ & $\min(p, q)$ & $\gdw p = q$ \\
$\mathrm{II}_n$ & D~III & $\factor{\mathrm{SO}^*(2n)}{\U n}$ & $\frac12 n(n - 1)$ & $\rd n/2.$ & $\gdw$ $n$ even \\
$\mathrm{III}_n$ & BD~I ($q = 2$) & $\factor{\mathrm{SO}_0(2, n)}{\SO 2 \x \SO n}$ & $n$ & $\min(n, 2)$ & yes \\
$\mathrm{IV}_{\!n}$ & C~I & $\factor{\mathrm{Sp}(2n, \R)}{\U n}$ & $\frac12 n(n + 1)$& $n$ & yes \\
$\mathrm V$ & E~III & $\factor{E_{6(-14)}}{\SO{10} \x \SO 2}$ & $16$ & 2 & no \\
$\mathrm{VI}$ & E~VII & $\factor{E_{7(-25)}}{E_{6(-78)} \x \SO 2}$ & $27$ & 3 & yes \\
\hline
\end{tabular}
\vspace{.1em}
\caption{Classification of irreducible Hermitian symmetric spaces}
\label{ihss classification}
\end{table}

\begin{prp} \label{rank one spaces}
The irreducible Hermitian symmetric spaces of rank one are exactly the unit balls $\bB^n$ with $n \ge 1$.
In particular, the only tube domain of rank one is the unit disc $\bB^1$.
\end{prp}

\begin{proof}
According to \cref{ihss classification}, the rank equals one exactly in the following cases:
\begin{itemize}
\item $\mathrm I_{1, q}$ with $q \ge 1$, which gives $\bB^q$
\item $\mathrm{II}_2$, which gives $\bB^1$
\item $\mathrm{II}_3$, which gives $\bB^3$~\cite[p.~519, item~(vii)]{Helgason78}
\item $\mathrm{III}_1$, which gives $\bB^1$
\item $\mathrm{IV}_1$, which gives $\bB^1$
\end{itemize}
The second statement follows because $\bB^n$ is not of tube type for $n \ge 2$ (also by \cref{ihss classification}).
\end{proof}

\begin{prp} \label{dim and K_0 uniqueness}
An irreducible Hermitian symmetric space $\factor{G_0}{K_0}$ is uniquely determined by its dimension and the Lie algebra $\frk_0$ of $K_0$.
\end{prp}

\begin{proof}
\begin{table}[t]
\centering
\begin{tabular}{|c|c|c|c|c|c|}
\hline
\cite{Cartan35} & \cite{Helgason78} & $\dim_\C \factor{G_0}{K_0}$ & $\dim_\R \frk_0$ & $[\frk_0, \frk_0]$ \\
\hline \hline
$\mathrm I_{p, q}$ & A~III & $pq$ & $p^2 + q^2 - 1$ & $\mathfrak{su}(p) \oplus \mathfrak{su}(q)$ \\
$\mathrm{II}_n$ & D~III & $\frac12 n(n - 1)$ & $n^2$ & $\mathfrak{su}(n)$ \\
$\mathrm{III}_n$ & BD~I ($q = 2$) & $n$ & $\frac12 n(n - 1) + 1$ & $\mathfrak{so}(n)$ \\
$\mathrm{IV}_{\!n}$ & C~I & $\frac12 n(n + 1)$ & $n^2$ & $\mathfrak{su}(n)$ \\
$\mathrm V$ & E~III & $16$ & $46$ & $\mathfrak{so}(10)$ \\
$\mathrm{VI}$ & E~VII & $27$ & $79$ & $\mathfrak e_6$ \\
\hline
\end{tabular}
\vspace{.5em}
\caption{Irreducible Hermitian symmetric spaces with dimensions and commutator subalgebra}
\label{dim and K_0 table}
\end{table}

From \cref{dim and K_0 table}, we see that $[\frk_0, \frk_0]$ is always semisimple (this is clear because $K_0$ is compact), and it is not simple if and only if we are in case~$\mathrm I_{p, q}$ with $p, q \ge 2$.
From this observation, one can check that the only potential ``collisions'' are the following:
\begin{itemize}
\item $\mathrm I_{1, q}$ and $\mathrm{II}_n$ with $q = n = \frac12 n(n - 1)$, i.e.~$n = 3$.
But $\mathrm{II}_3$ is the unit ball $\bB^3$ (see above), so they are actually isomorphic.
\item $\mathrm I_{1, q}$ and $\mathrm{IV}_n$ with $q = n = \frac12 n(n + 1)$, i.e.~$n = 1$.
But $\mathrm{IV}_1$ is the unit ball $\bB^1$.
\item $\mathrm{II}_4$ and $\mathrm{III}_6$, which are isomorphic~\cite[p.~519, item~(viii)]{Helgason78}.
\item $\mathrm{III}_3$ and $\mathrm{IV}_2$, which are isomorphic~\cite[p.~519, item~(ii)]{Helgason78}.
\end{itemize}
This proves the proposition.
\end{proof}

\section{Compactness of the holonomy group}

In this section we prove \cref{hol cover intro} (as \cref{hol cover}).
First we set up some notation.

\begin{setup} \label{setup holonomy}
Let $X$ be an $n$-dimensional normal projective variety with klt singularities and ample canonical divisor $K_X$.
Let $\omega_\ke$ be as in~\cite[Setup~3.2]{GrafPatel25}.
Write $g_\ke$ for the associated Riemannian metric on $\reg X$ and $h_\ke$ for the associated Hermitian metric on $\T{\reg X}$.
Fix, once and for all, a smooth point $x \in \reg X$, and consider the tangent space $V \defn T_x X$ at that point.
We write
\begin{align*}
H \defn \operatorname{Hol}(\reg X, g_\ke, x) & \subset \U{V, h_{\ke, x}} \isom \U n \quad \text{and} \\
H^\circ \defn \operatorname{Hol}^\circ(\reg X, g_\ke, x) & \subset \U{V, h_{\ke, x}}
\end{align*}
for the corresponding (restricted) holonomy group.
Recall from~\cite[Cor.~10.41]{Besse87} and~\cite[Lemma~5.2]{GrafPatel25} that there are decompositions
\begin{equation} \label{canonical decomposition}
V = V_1 \oplus \cdots \oplus V_m \quad \text{and} \quad H^\circ = H_1^\circ \x \cdots \x H_m^\circ
\end{equation}
such that for each $1 \le i \le m$, the factor $H_i^\circ$ acts irreducibly and non-trivially on~$V_i$ and trivially on $V_j$ for $j \ne i$.
\end{setup}

\begin{thm}[Holonomy cover] \label{hol cover}
Let $X$ be as above.
\begin{enumerate}
\item The holonomy group $H$ of $(\reg X, \omega_\ke)$ is compact.
\item There is a finite \qe Galois cover $\gamma \from Y \to X$ such that the holonomy group of $(\reg Y, \omega_{\ke, Y})$ is connected.
\end{enumerate}
\end{thm}

\begin{proof}
It is sufficient to show that the index of the identity component $H^\circ$ in $H$ is finite, i.e.~$[H : H^\circ] < \infty$.
For then $H$ is clearly compact ($H^\circ$ being compact), and the kernel of the natural surjection $\pi_1(\reg X) \surj \factor HH^\circ$ will yield the desired finite cover~$\gamma$.

In order to show that $[H : H^\circ] < \infty$, it is sufficient to show that the normalizer $N_{\U V}(H^\circ)$ is a finite extension of $H^\circ$, since $H \subset N(H^\circ)$.
By~\cite[Lemma~5.3]{GrafPatel25}, it is in turn sufficient to show that $N_{\U{V_i}}(H_i^\circ)$ is a finite extension of $H_i^\circ$ for each $1 \le i \le m$.
We will therefore assume from now on that $m = 1$, i.e.~that $H^\circ$ acts irreducibly on $V$.

According to~\cite[Cor.~10.92]{Besse87}, either $H^\circ = \U V$ or $(\reg X, g_\ke)$ is locally symmetric.
In the former case, the claim is clear, so we may assume that we are in the second case.
That is, $H^\circ$ is the holonomy of an irreducible simply connected Hermitian symmetric space $\factor G{H^\circ}$.

The key fact in dealing with this situation is that the center $Z(H^\circ)$ is exactly one-dimensional.
This follows from~\cite[App.~5, Lemma 2(3)]{KobayashiNomizu96} and~\cite[Ch.~XI, Thm.~9.6(1)]{KobayashiNomizu96II}.
But it is also possible to verify this claim case by case, using the classification of symmetric spaces~\cite[Table~3 on p.~315]{Besse87}.

By Schur's lemma, we have
\[ Z(H^\circ) = \U1 \cdot \id_V \defn \set{ \lambda \cdot \id_V \;\;\big|\;\; |\lambda| = 1 } \subset \U V. \]
Let $C = C_{\U V}(H^\circ)$ be the centralizer of $H^\circ$.
By \cref{Out finite} and \cref{N/H ses} below, it is sufficient to show that $\factor C{Z(H^\circ)}$ is finite.
But by Schur's lemma again, $C = \U1 \cdot \id_V$.
So $\factor C{Z(H^\circ)}$ is even trivial.
\end{proof}

\begin{prp} \label{Out finite}
Let $G$ be a compact connected Lie group. Then the outer automorphism group
\[ \Out(G) \defn \factor{\Aut(G)}{\Inn(G)} \]
is finite if and only if the center $Z(G)$ has dimension at most one.
\end{prp}

\begin{proof}
Let $k \defn \dim Z(G)$.
Since $G$ is compact and connected, it has a finite cover $p \from \wt G \to G$ with
\[ \wt G = G' \x T, \]
where $G'$ is a compact connected semisimple Lie group and $T \isom \bT^k \defn (S^1)^k$ is a $k$-dimensional torus~\cite[Thm.~4.29]{Knapp02}.
The kernel of $p$ is a finite central subgroup $D \subset Z(\wt G)$.

\medskip \noindent
``$\Rightarrow$'': Suppose that $k \ge 2$.
Recall that an automorphism $\phi \in \Aut(\wt G)$ descends to an automorphism of $G$ if and only if it stabilizes $D$, i.e.~if $\phi(D) = D$.

\begin{clm} \label{Aut_D(T)}
The subgroup
\[ \Aut_D(T) \defn \set{ \phi \in \Aut(T) \;\;\big|\;\; \text{$\id_{G'} \x \phi$ stabilizes $D$} } \]
has finite index in $\Aut(T)$.
\end{clm}

\begin{proof}
Let $D_T$ denote the projection of $D \subset \wt G$ to the second factor $T$.
Consider the subgroup
\[ H \defn \set{ \phi \in \Aut(T) \;\;\big|\;\; \text{$\phi$ fixes $D_T$ pointwise} }. \]
Clearly, $H \subset \Aut_D(T)$ and hence it suffices to show that $H$ has finite index in $\Aut(T)$.
Since there are only finitely many torsion points of any given order in $T$, fixing a torsion point defines a finite index subgroup of $\Aut(T)$ by the orbit-stabilizer theorem.
But $D_T$ consists of finitely many torsion points of $T$.
Therefore $H$ is a finite intersection of finite index subgroups, hence itself has finite index.
\end{proof}

The automorphism group of $T$ is $\Aut(T) = \GL k\Z$, which is infinite since $k \ge 2$.
By \cref{Aut_D(T)}, also $\Aut_D(T)$ is infinite.
We define a map
\begin{equation} \label{444}
\Aut_D(T) \lto \Out(G)
\end{equation}
by sending $\phi \in \Aut_D(T)$ to the coset of $\phi_G$, the automorphism of $G$ induced by $\id_{G'} \x \phi$.
If $\phi_G$ is an inner automorphism, i.e.~$\phi_G = \operatorname{Int}_g$ for some $g \in G$, then so is $\id_{G'} \x \phi = \operatorname{Int}_{\wt g}$, where $\wt g \in p\inv(g)$.
Since inner automorphisms of $\wt G$ act trivially on $T$, we get $\phi = \id_T$.
This shows that the map~\lref{444} is injective, and hence $\Out(G)$ is infinite, as desired.

\medskip \noindent
``$\Leftarrow$'': Suppose $k \le 1$.
Write $p \from \wt G = G' \x T \to G$ as before.
Since $G'$ is compact and semisimple, its fundamental group $\pi_1(G')$ is finite~\cite[Thm.~4.69]{Knapp02}.
We may and will therefore assume that $G'$ is simply connected.

\begin{clm} \label{Out cover}
There is an injective map $\Out(G) \to \Out(\wt G)$.
\end{clm}

\begin{proof}
First we show that any $\phi \in \Aut(G)$ lifts (uniquely) to some $\wt\phi \in \Aut(\wt G)$.
For this, it is sufficient to show that the subgroup $\pi_1(\wt G) \subset \pi_1(G)$ is stabilized by any such $\phi$.
Note that $\pi_1(\wt G) = \pi_1(T) \isom \Z^k$, so we may assume that $k = 1$.
Let $T_G = p(T) \subset G$, then $T_G \isom S^1$ and it suffices to show that $T_G$ is stabilized by $\phi$.
Since $Z(G')$ is finite, $T = Z(\wt G)^\circ$ is the identity component of the center of $\wt G$.
This implies that $T_G = Z(G)^\circ$ because $p\inv(Z(G)) = Z(\wt G)$.
But $Z(G)^\circ$ is invariant under any automorphism of $G$.

Sending $\phi \mapsto \wt\phi$ defines a map $\Aut(G) \to \Aut(\wt G)$.
If $\phi = \operatorname{Int}_g$ is inner, then so is the lift $\wt\phi = \operatorname{Int}_{\wt g}$, where $\wt g \in p\inv(g)$.
Conversely, if $\wt\phi = \operatorname{Int}_{\wt g}$ is inner, then also $\phi = \operatorname{Int}_{p(g)}$ is inner.
Therefore the lifting map $\Aut(G) \to \Aut(\wt G)$ induces a map $\Out(G) \to \Out(\wt G)$, and the latter is injective.
\end{proof}

Recall that we want to show that $\Out(G)$ is finite.
By \cref{Out cover}, it is sufficient to show that $\Out(\wt G)$ is finite.

\begin{clm} \label{Aut cover}
We have $\Aut(\wt G) = \Aut(G') \x \Aut(T)$.
\end{clm}

\begin{proof}
Let $\phi \in \Aut(\wt G)$.
Then $\phi$ is of the form
\[ \phi(g, z) = \big( \alpha(g) \cdot \gamma(z), \, \delta(g) \cdot \beta(z) \big), \]
where $\alpha \from G' \to G'$, $\beta \from T \to T$, $\gamma \from T \to G'$ and $\delta \from G' \to T$.
We must show that $\gamma$ and $\delta$ are trivial.

Since $G'$ is semisimple, it is equal to its own commutator subgroup and so $\delta$ is trivial~\cite[Lemma~4.28]{Knapp02}.
Now consider $\gamma \from T \to G'$.
The image of $\gamma$ must lie inside $C_{G'}(\alpha(G'))$, the centralizer of the image of $\alpha$.
But $\phi(G' \x \set1) \subset G' \x \set1$ because $\delta$ is trivial.
The same argument can be applied to $\phi\inv$, so actually $\phi(G' \x \set1) = G' \x \set1$.
On the other hand, $\phi(G' \x \set1) = \img(\alpha) \x \set1$, so $\alpha$ is surjective.
So the image of $\gamma$ lies in $Z(G')$, which is finite.
But $T$ is connected, thus $\gamma$ is also trivial.
\end{proof}

Clearly, $\Inn(G' \x T) = \Inn(G') \x \Inn(T)$.
By \cref{Aut cover}, we get
\[ \Out(\wt G) = \frac{\Aut(G') \x \Aut(T)}{\Inn(G') \x \Inn(T)}
= \Out(G') \x \Out(T). \]
Since $G'$ is semisimple and simply connected, $\Out(G')$ is finite: outer automorphisms correspond to the (finitely many) automorphisms of the Dynkin diagram of~$G'$, cf.~\cite[Thm.~7.8]{Knapp02}.
Likewise, $\Out(T) = \Aut(T) = \GL k \Z$ is finite because $k \le 1$.
So $\Out(\wt G)$ is finite, too.
\end{proof}

\begin{prp} \label{N/H ses}
Let $G$ be a group, $H \subset G$ a subgroup, $N = N_G(H)$ the normalizer and $C = C_G(H)$ the centralizer of $H$ in $G$.
Then there is an exact sequence
\[ 1 \lto \factor C{Z(H)} \lto \factor NH \lto \Out(H). \]
\end{prp}

\begin{proof}
Each $n \in N$ defines an automorphism of $H$ by conjugation:
\[ \phi(n)(h) = n h n\inv. \]
Then $\phi \from N \to \Aut(H)$ is a group homomorphism and $\phi(H) \subset \Inn(H)$.
Therefore $\phi$ induces a map
\[ \psi \from \factor NH \lto \Out(H). \]
An element $n \in N$ maps to the identity in $\Out(H)$ if and only if $\phi(n) \in \Inn(H)$ if and only if there exists $h \in H$ such that
\[ n x n\inv = h x h\inv \quad \text{for all $x \in H$} \]
if and only if $h\inv n \in C$ if and only if $n \in hC$.
Thus $\phi\inv \big( \!\Inn(H) \big) = HC$ (note that this is a subgroup).
Therefore the kernel of $\psi$ is
\[ \ker(\psi) = \factor{HC}H \isom \factor C{C \cap H} = \factor C{Z(H)}. \]
This ends the proof.
\end{proof}

\section{The restricted holonomy group} \label{sec Hol0}

In this section we will deduce information about the restricted holonomy group $H^\circ$ from the existence of certain tensors on $X$, using the Bochner principle.
As in~\cite[Lemma~5.4]{GrafPatel25}, the proofs are modeled on arguments in~\cite{CataneseDiScala13, CataneseDiScala14}.
The only difference is that instead of considering the universal cover of $X$ (which we clearly cannot do), we argue purely locally.

\begin{lem} \label{main2 Hol0}
Assume that $X$ satisfies condition~\lref{main2.3}, i.e.~there is a semispecial tensor $\psi$ on $X$.
Then each $H_i^\circ$ is the holonomy of an irreducible bounded symmetric domain $\cD_i$ of tube type whose rank divides its dimension.
Moreover, each $\cD_i$ is determined uniquely by the tensor $\psi$.
\end{lem}

\begin{proof}
Using notation from \cref{setup holonomy}, consider the decomposition
\[ V = U_1' \oplus U_1'' \oplus U_2 \]
where $U_1'$ is the sum of all the $V_i$ such that the corresponding $H_i^\circ$ is the holonomy of a bounded symmetric domain of tube type, $U_1''$ is the sum of the $V_i$ such that $H_i^\circ$ is the holonomy of a bounded symmetric domain which is neither a ball nor of tube type, and $U_2$ is the sum of the $V_i$ such that $H_i^\circ$ acts transitively.

Let $u_1, \dots, u_a; w_1, \dots, w_b; z_1, \dots, z_r$ be coordinates on $U_1'$, $U_1''$ and $U_2$, respectively.
Let $\vol_1', \vol_1'', \vol_2$ be the corresponding volume forms.
The tensor $\psi$ evaluated at $x$ can be written as
\[ \psi_x = f(u, w, z) \cdot (\vol_1')\inv \wedge (\vol_1'')\inv \wedge (\vol_2)\inv, \]
where $f$ is a homogeneous degree $n$ polynomial on $V$.

\begin{clm}
The polynomial $f$ only depends on $u$, i.e.~$f(u, w, z) = f(u)$.
\end{clm}

\begin{proof}
By the Bochner principle~\cite[Cor.~3.4]{GrafPatel25} and the holonomy principle, $\psi_x$ is $H^\circ$-invariant.
In particular, its zero scheme $\set{ \psi_x = 0 } \subset \bP V$ is also $H^\circ$-invariant.
But $\set{ \psi_x = 0 } = \set{ f = 0 }$, so we can apply~\cite[Prop.~A.1]{CataneseFranciosi09} to $f$.
We obtain that $f$ does not depend on $z$.
Then by~\cite[Cor.~2.2]{CataneseDiScala13}, $f$ also does not depend on~$w$.
\end{proof}

We can now argue as in~\cite[proof of Thm.~1.2, p.~428]{CataneseDiScala13} to obtain that $U_1'' = U_2 = 0$, that is, each $H_i^\circ$ is the holonomy of a bounded symmetric domain of tube type.
It remains to prove the second part of the statement.

Pick a bounded symmetric domain $0 \in \cD \subset \C^n$ such that the action of $H^\circ$ on $V = U_1'$ equals the action of $K^\circ$ on $T_0 \cD$, where $K \subset \Aut(\cD)$ is the stabilizer of $0 \in \cD$.
By the above, we may assume that $\cD = \cD_1 \x \cdots \x \cD_m$ is a product of bounded symmetric domains of tube type.
Again, as in~\cite[p.~428]{CataneseDiScala13} it follows that $\rk \cD_j$ divides $\dim \cD_j$ for each $j = 1, \dots, m$.

To prove the second claim, we may use the classification~\cite[Thm.~2.3]{CataneseDiScala13} as is, to conclude that the pairs $(\rk \cD_j, \dim \cD_j)$ uniquely determine each domain $\cD_j$.
\end{proof}

\begin{lem} \label{main3 Hol0}
Assume that $X$ satisfies condition~\lref{main3.3}, i.e.~there is a slope zero tensor $\psi$ on $X$.
Then each $H_i^\circ$ is the holonomy of an irreducible bounded symmetric domain $\cD_i$ of tube type.
\end{lem}

\begin{proof}
In the smooth case, the argument is very similar to the case of semispecial tensors, cf.~\cite[proof of Thm.~1.3, p.~436]{CataneseDiScala13}.
The changes that need to be made in our singular setting, on the other hand, are similar to those in the proof of \cref{main2 Hol0}.
Therefore, we will omit the details.
\end{proof}

\begin{lem} \label{main4 Hol0}
Assume that $X$ satisfies condition~\lref{main4.3}.
Then each $H_i^\circ$ is the holonomy of an irreducible bounded symmetric domain $\cD_i$ not isomorphic to $\bB^m$ for any $m \ge 2$.
Moreover, each $\cD_i$ is determined uniquely by the tensor $\sigma$.
\end{lem}

\begin{proof}
Since $K_X$ is assumed to be ample, the tensor $\sigma$ is parallel with respect to the \kahler--Einstein metric $\omega_\ke$ on $\reg X$ by the Bochner principle~\cite[Cor.~3.4]{GrafPatel25}.
Thus the decomposition of the first Mok characteristic cone $\mathcal{CS}$ of $\sigma$ into irreducible components is invariant under the action of $H^\circ$.

Since $K_X$ is ample, each $H_i^\circ$ is either equal to $\U{V_i}$, or $H_i^\circ$ acts on $V_i$ as the holonomy of an irreducible bounded symmetric domain $\cD_i$ of rank $\ge 2$.
Since we work locally around the smooth point $x \in \reg X$, the arguments in the proof of~\cite[Thm.~1.2]{CataneseDiScala14} work verbatim.
In particular, they show that the cones $\mathcal{CS}_i$ are just the origin when $\dim V_i = 1$, and they are the cones over the first Mok characteristic variety otherwise.
This rules out the case $H_i^\circ = \U{V_i}$, for all $i$.
In particular, no $\cD_i$ can be a ball $\bB^m$ with $m \ge 2$.

The tensor $\sigma$ determines the first Mok characteristic varieties $\cS_i^1$ (see~\cite[Def.~2.1]{CataneseDiScala14}), and since $H_i^\circ = \operatorname{Hol}(\cD_i, \omega_{\mathrm{Berg}})$ for all $i$, where the $\cD_i$ are irreducible bounded symmetric domains, we have $\cS_i^1 = \cS^1(\cD_i)$ for all $i$.
From the claim in the proof of~\cite[Thm.~1.2]{CataneseDiScala14}, each irreducible bounded symmetric domain $\cD_i$ is determined by the data $(\dim \cD_i, \dim \cS^1(\cD_i))$.
\end{proof}

\section{Local isometries}

Assume that we are in the situation of \cref{main2}, \cref{main3} or \cref{main4}, and that we want to prove the implication from the second item (existence of a tensor) to the first one (uniformization).
Applying the corresponding lemma from \cref{sec Hol0}, we get irreducible bounded symmetric domains $\cD_i$ such that each $H_i^\circ$ is the holonomy of $\cD_i$.
Set $\cD \defn \cD_1 \x \cdots \x \cD_m$ and let $\omega_{\mathrm{Berg}}$ be the Bergman metric on $\cD$.

\begin{prp} \label{local iso}
The spaces $(\reg X, \omega_\ke)$ and $(\cD, \omega_{\mathrm{Berg}})$ are locally isometric.
\end{prp}

\begin{proof}
Pick an arbitrary point $x \in \reg X$ and let $x \in U \subset X$ be a sufficiently small open neighborhood such that the decomposition~\lref{canonical decomposition} of $T_x X$ is induced by a decomposition $U = U_1 \x \cdots \x U_m$, cf.~\cite[Thm.~10.38]{Besse87}.
We know that for each $i = 1, \dots, m$, the manifolds $U_i$ and $\cD_i$ have the same holonomy group, namely $H_i^\circ$, and we want to show that they are locally isometric.
After renumbering, we may assume that there is an $m_0$ such that $\dim U_i = 1$ if and only if $i \le m_0$.

If $i \le m_0$, then $U_i$ is locally isometric to the unit disc with the Poincar\'e metric, which is nothing but $\cD_i$.
Otherwise, $\cD_i$ is a symmetric space of rank $\ge 2$: indeed, by \cref{rank one spaces}, if $\rk \cD_i = 1$, then $\cD_i$ would be a ball $\bB^n$ with $n \ge 2$, which is excluded in each case by the corresponding lemma in \cref{sec Hol0}.

By~\cite[Thm.~10.90]{Besse87}, the holonomy of $\cD_i$ is \emph{not} transitive on the unit sphere of its tangent space at the base point.
Therefore also the holonomy of $U_i$ is not transitive on the unit sphere.
By said theorem again, $U_i$ is an irreducible locally symmetric space of rank $\ge 2$.
\cref{dim and K_0 uniqueness} then implies that $U_i$ and $\cD_i$ are locally isometric.
\end{proof}

\section{Proof of \cref{main5}}

In this section we prove \cref{main5}.
In \cref{tensor to uvhs}, we show how to obtain a uVHS from a tensor whose first Mok characteristic cone satisfies the properties in the theorem.
In \cref{uvhs to tensor}, we deal with the converse direction.

\begin{rem} \label{H = K}
By~\cite[Prop.~10.79]{Besse87}, each irreducible bounded symmetric domain $\cD_i$ in \cref{main4 Hol0} can be expressed as $\cD_i = \factor{G_i^0}{K_i^0}$, where $K_i^0 = H_i^\circ = \operatorname{Hol}(\cD_i, \omega_{\mathrm{Berg}})$ and $H_i^\circ$ acts on the tangent space $T_e \cD_i$ via the adjoint representation.
It follows that the irreducible holonomy factors $H_i^\circ$, and hence the restricted holonomy group $H^\circ$, are also determined by $\sigma$.
\end{rem}

\begin{lem} \label{uvhs cover}
Let $\gamma \from Y \to X$ be a finite \qe cover, where $X$ and $Y$ are normal projective varieties with klt singularities.
If $\reg Y$ admits a uniformizing VHS for some Hodge group $G_0$ of Hermitian type then so does $\reg X$.
\end{lem}

\begin{proof}
Suppose $\reg Y$ admits a uniformizing VHS corresponding to $G_0$.
Then we know from \cref{prp uVHS} that $Y$ is uniformized by $\cD = \factor{G_0}{K_0}$.
Set $Y^\circ \defn \gamma\inv(\reg X)$.
Since $Y^\circ \to \reg X$ is \'etale, $\reg X$ and $Y^\circ$ have the same universal cover~$\cD^\circ$, which is a big open subset of $\cD$.
Moreover, $\reg X \isom \factor{\cD^\circ}{\Gamma}$, where $\Gamma \subset \Aut \cD$ is a discrete cocompact subgroup isomorphic to $\pi_1(\reg X)$.

According to \cref{bsd uvhs}, there is a uniformizing VHS $(P_\cD, \theta_\cD)$ on $\cD$ corresponding to $G_0$, which restricts to a uniformizing VHS $(P_\cD, \theta_\cD)\big|_{\cD^\circ}$ on $\cD^\circ$.
Following the same arguments as in the proof of \cref{prp uVHS}, we see that $(P_\cD, \theta_\cD)\big|_{\cD^\circ}$ is $\Gamma$-equivariant and descends to a uniformizing VHS on $\reg X$.
\end{proof}

We can now prove the implication ``$\lref{main5.2} \imp \lref{main5.1}$''.

\begin{prp} \label{tensor to uvhs}
If $X$ admits a holomorphic tensor $\sigma$ as in \cref{main5}, then there is a uniformizing VHS $(P, \theta)$ on $\reg X$ for a Hodge group $G_0$ of Hermitian type such that $\frg_0$ has no factors isomorphic to $\mathfrak{su}(p, 1)$ with $p \ge 2$.
\end{prp}

\begin{proof}
Using notation from \cref{setup holonomy}, we know by \cref{hol cover} that $H$ is compact and that there is a \qe Galois cover $\gamma \from Y \to X$ such that $\operatorname{Hol}(\reg Y, g_{\ke, Y})$ is connected.
By \cref{uvhs cover}, $\reg X$ admits a uniformizing VHS if $\reg Y$ does.
Therefore, we may assume without loss of generality that $H = H^\circ$ is connected.

\begin{clm} \label{ushb Xreg}
There is a uniformizing system of Hodge bundles $(P, \theta)$ on $\reg X$, together with a metric $P_H \subset P$.
\end{clm}

\begin{proof}[Proof of \cref{ushb Xreg}]
Recall that any hermitian vector bundle admits a reduction of structure group to its holonomy group.
In particular, the tangent bundle of $\reg X$ can be written as $\T{\reg X} \isom P_H \times_H V$, where $P_H$ is the holonomy subbundle of the Chern connection on $\T{\reg X}$, the vector space $V$ is the tangent space $T_x X$ as in \cref{setup holonomy}, and $H$ acts on $V$ via the holonomy representation.

From \cref{main4 Hol0} we have a decomposition $H = H^\circ = H_1^\circ \x \cdots \x H_m^\circ$, where $H_i^\circ = \operatorname{Hol}(\cD_i, \omega_{\mathrm{Berg}})$ for $\cD_i$ an irreducible bounded symmetric domain not isomorphic to a higher-dimensional ball for each $i = 1, \dots, m$.
Recall that $V$ decomposes as $V = V_1 \oplus \cdots \oplus V_m$, where $H_i^\circ$ acts irreducibly on $V_i$ and trivially on $V_j$ for $j \ne i$.

By \cref{H = K}, we have $\cD_i = \factor{G_i^0}{K_i^0}$, where $G_i^0$ is a Hodge group of Hermitian type.
Let $G_i$ and $K_i$ denote the complexifications of $G_i^0$ and $K_i^0$, respectively.
Then $\frg_i \defn \operatorname{Lie}(G_i)$ admits a Hodge decomposition given by $\frg_i = \frg_i^{-1,1} \oplus \frg_i^{0,0} \oplus \frg_i^{1,-1}$.
By \cref{bsd uvhs}, the tangent bundle of $\cD_i$ can be written as
\[ \T{\cD_i} \isom P_i \x_{K_i^0} \frg_i^{-1,1}, \]
where $P_i$ is the principal $K_i^0$-bundle $G_i^0 \to \cD_i$ and $K_i^0$ acts on $\frg_i^{-1,1}$ via the adjoint representation.
By~\cite[Prop.~10.79]{Besse87}, the given action of $H_i^\circ$ on $V_i$ is isomorphic to the adjoint action of $K_i^0$ on $\frg_i^{-1,1}$ under the above identifications.
Since this action is via \C-linear maps, the adjoint representation of $K_i^0$ on $V_i$ factors through the complexification $K_i$.
It then follows that $V_i$ and $\frg_i^{-1,1}$ are isomorphic as $K_i$-representations.

Let $\cD \defn \cD_1 \x \cdots \x \cD_m = \factor{G_0}{K_0}$, where $G_0 \defn G_1^0 \x \cdots \x G_m^0$ and $K_0 \defn K_1^0 \x \cdots \x K_m^0$.
Note that $\frg_0 = \operatorname{Lie}(G_0)$ has no factors isomorphic to $\mathfrak{su}(p, 1)$ with $p \ge 2$ because no factor $\cD_i$ is a higher-dimensional ball.
Set $\frg \defn \operatorname{Lie}(G)$ with $G$ the complexification of $G_0$ and $K$ the complexification of $K_0$.
Consider the principal $K$-bundle $P \defn P_H \x_{K_0} K$ on $\reg X$, where we have identified $H$ and $K_0$.
By the above discussion, we have a chain of isomorphisms
\[ \T{\reg X} \isom P_H \x_{K_0} V \isom P \x_K V \isom \bigoplus_{i=1}^m P \x_{K_i} V_i \isom \bigoplus_{i=1}^m P \x_{K_i} \frg_i^{-1,1} \isom P \x_K \frg^{-1,1}. \]
This yields an isomorphism
\[ \theta \from \T{\reg X} \lto P \x_K \frg^{-1,1} \]
of vector bundles, which automatically satisfies $[ \theta(u), \theta(v) ] = 0$ for all local sections $u, v$ of $\T{\reg X}$.
Therefore, $(P, \theta)$ is a uniformizing system of Hodge bundles on $\reg X$, and $P_H \subset P$ is a metric.
\end{proof}

\begin{clm} \label{uvhs Xreg}
The uniformizing system of Hodge bundles $(P, \theta)$ together with the metric $P_H$ given in \cref{ushb Xreg} is in fact a uniformizing VHS.
\end{clm}

\begin{proof}[Proof of \cref{uvhs Xreg}]
On the bounded symmetric domain $\cD$ from above, there is by \cref{bsd uvhs} a natural uniformizing system of Hodge bundles $(P_\cD, \theta_\cD)$ given by $P_\cD \defn G_0 \x_{K_0} K$ and $\theta_\cD \from \T \cD \bij P_\cD \x_K \frg^{-1,1}$.
The principal $K$-bundle $P_\cD$ admits a metric $P_{\cD, H}$ given by $G_0 \subset P_\cD$.
This makes $(P_\cD, \theta_\cD)$ together with $P_{\cD, H}$ into a uniformizing VHS.

It follows from \cref{local iso} that $\reg X$ and the domain $\cD$ are locally isometric.
More precisely, any point $p \in \reg X$ has an analytic open neighbourhood $U$ such that there is an isometry $f \from U \bij U'$, where $U'$ is an analytic open subset of $\cD$.
Let $F_U$ be the frame bundle of $\T U$ and $F_{U'}$ the frame bundle of $\T{U'}$.
Then $f$ induces an isomorphism $\d f \from F_U \bij f^* F_{U'}$.
Since $f$ is an isometry, $\d f$ preserves the holonomy subbundles of $U$ and $U'$, that is, $\d f \big( P_H\big|_U \big) = f^* \big( P_{\cD, H}\big|_{U'} \big)$ as subbundles of $f^* F_{U'}$.
Extending the structure group from $K_0$ to $K$, we see that $P\big|_U$ and $f^* \big( P_\cD\big|_{U'} \big)$ are also isomorphic.
We then get a commutative diagram
\[ \begin{tikzcd}
\T U \arrow[rr, "\theta"] \arrow[d, "\wr", swap] & & P\big|_U \x_K \frg^{-1,1} \arrow[d, "\wr"] \\
f^* \T{U'} \arrow[rr, "f^* \theta_\cD"] & & f^* \big( P_\cD\big|_{U'} \big) \x_K \frg^{-1,1}.
\end{tikzcd} \]
This means that the uniformizing systems of Hodge bundles $(P, \theta)$ and $(P_\cD, \theta_\cD)$ are locally isomorphic via $\d f$.
Furthermore, as noted above the given metrics $P_H$ and $P_{\cD, H}$ are compatible with this isomorphism.

Let $R \defn P_H \x_{K_0} G_0$ be equipped with the connection $D_h$ from \cref{flat metric}, and analogously define $R_\cD \defn P_{\cD, H} \x_{K_0} G_0$ with the connection $D_\cD$.
By construction, $(R, D_h)$ is determined by the data given by $(P, \theta)$ and $P_H$, and analogously $(R_\cD, D_\cD)$ is determined by $(P_\cD, \theta_\cD)$ and $P_{\cD, H}$.
It therefore follows from the above observation that there is an isomorphism $(R, D_h)\big|_U \isom f^* \big( (R_\cD, D_\cD)\big|_{U'} \big)$.

Recall that in order to prove the claim, we need to show that $D_h$ is a flat connection.
Since we already know that $D_\cD$ is flat, it follows from the above isomorphism that $D_h\big|_U$ is flat.
But flatness is a local property, hence it follows by varying $p \in \reg X$ that $D_h$ is in fact flat.
Thus $(P, \theta)$ is a uniformizing VHS on~$\reg X$.
\end{proof}
This ends the proof of \cref{tensor to uvhs}.
\end{proof}

Now we prove the converse implication ``$\lref{main5.1} \imp \lref{main5.2}$''.
This will complete the proof of \cref{main5}.

\begin{prp} \label{uvhs to tensor}
Let $X$ be as in \cref{main5}, and suppose that $X$ admits a uniformizing VHS $(P, \theta)$ corresponding to a Hodge group $G_0$ of Hermitian type such that $\frg_0$ has no factors isomorphic to $\mathfrak{su}(p, 1)$ with $p \ge 2$.
Then there is a holomorphic tensor $\sigma \in \HH0.X.\sE nd \big( \T X [\tensor] \Omegap X1 \big).$ satisfying properties~\lref{p1}--\lref{p3}.
\end{prp}

\begin{proof}
By assumption, there is a bounded symmetric domain $\cD = \factor{G_0}{K_0}$ without higher-dimensional ball factors and an isomorphism $\theta \from \T{\reg X} \bij P \x_K \frg^{-1,1}$.
Dualizing, we get $\Omegap{\reg X}1 \isom P \x_K \frg^{1,-1}$.
Let $\cD = \cD_1 \x \cdots \x \cD_m$ be the decomposition of $\cD$ into irreducible bounded symmetric domains $\cD_i = \factor{G_i^0}{K_i^0}$.

The Lie algebra $\frg$ decomposes into simple factors $\frg = \bigoplus_{i=1}^m \frg_i$, where $\frg_i$ is the complexified Lie algebra of $G_i^0$.
From this we get $\frg^{p,-p} = \bigoplus_{i=1}^m \frg_i^{p,-p}$ for all $p \in \set{ -1, 0, 1 }$.
The Lie bracket of each $\frg_i$ induces a surjective morphism
\[ \rho_i \from \frg_i^{-1,1} \tensor \frg_i^{1,-1} \lto \frg_i^{0,0}, \qquad u \tensor v \mapsto [u, v]. \]
There is also a natural injective morphism in the other direction:
\[ \tau_i \from \frg_i^{0,0} \lto \frg_i^{-1,1} \tensor \frg_i^{1,-1} = \Hom \frg_i^{-1,1}.\frg_i^{-1,1}., \qquad w \mapsto [w, -]. \]
Setting $\sigma'_i \defn \tau_i \circ \rho_i$, we obtain an endomorphism of $\frg_i^{-1,1} \tensor \frg_i^{1,-1}$.
For each $1 \le i \le m$, we define $\sigma_i$ as follows:
\[ \sigma_i \defn \begin{cases}
\sigma'_i & \text{if $\dim\cD_i > 1$,} \\
\id_{\frg_i^{-1,1} \tensor \frg_i^{1,-1}} & \text{if $\dim\cD_i = 1$.}
\end{cases} \]
Finally we set $\sigma \defn \bigoplus_{i=1}^m \sigma_i$, an endomorphism of $\frg^{-1,1} \tensor \frg^{1,-1}$.

The endomorphism $\sigma$ induces an endomorphism of $\T{\reg X} \tensor \Omegap{\reg X}1$, which by reflexivity extends to an endomorphism $\sigma \in \HH0.X.\sE nd \big( \T X [\tensor] \Omegap X1 \big).$.
It remains to verify that $\sigma$ satisfies the three properties \lref{p1}--\lref{p3}, for which we restrict to any point $p \in \reg X$ and work again with $\sigma \in \End \frg^{-1,1} \tensor \frg^{1,-1}.$.

By construction, $\sigma_i$ coincides with the algebraic curvature tensor of an irreducible bounded symmetric domain defined in~\cite[p.~211]{KobayashiOchiai81} (paragraph preceding Lemma 2.9), for all $i$ such that the rank of $\cD_i$ is $\ge2$.
It is then clear that $\sigma$ coincides with the tensor defined in the first part of the proof of~\cite[Thm.~1.2]{CataneseDiScala14}.
Therefore, $\sigma$ satisfies \lref{p1}--\lref{p3} by the same arguments therein.
\end{proof}

\section{Proof of \cref{main2}, \cref{main3} and \cref{main4}}

\begin{proof}[Proof of \cref{main2}]
We prove both directions separately.

``$\lref{main2.1} \imp \lref{main2.3}$'': We know from~\cite[proof of Thm.~1.2, p.~429]{CataneseDiScala13} that on each $\cD_i$, there is a special tensor $\Psi_i$ semi-invariant under $\Aut(\cD_i)$.
(Here a special tensor is a semispecial tensor with $\eta$ the trivial line bundle, and ``semi-invariant'' means invariant up to a character $\chi_i \from \Aut(\cD_i) \to \set{ \pm 1}$.)
Consider the special tensor $\Psi \defn \operatorname{pr}_1^* \Psi_1 \tensor \cdots \tensor \operatorname{pr}_m^* \Psi_m$ on $\cD$, where $\operatorname{pr}_i \from \cD \to \cD_i$ are the projections.
As in~\cite[proof of Thm.~6.1]{GrafPatel25}, it can be seen that $\Psi$ descends to a semispecial tensor $\psi$ on $X$.

``$\lref{main2.3} \imp \lref{main2.1}$'': By \cref{main2 Hol0}, the existence of a semispecial tensor $\psi$ implies that $\operatorname{Hol}(\reg X, \omega_\ke)$ is the holonomy of a bounded symmetric domain $\cD$ whose irreducible factors are domains of tube type whose rank divides their dimension.
Moreover, from \cref{local iso} we have that $\reg X$ and $\cD$ are locally isometric.
By the same arguments as in the proof of \cref{tensor to uvhs}, it follows that $\reg X$ admits a uniformizing VHS for $\Aut(\cD)$.
We conclude from \cref{prp uVHS} that $X$ is uniformized by $\cD$.
\end{proof}

\begin{proof}[Proof of \cref{main3}]
The proof of \cref{main3} is very similar to the one of \cref{main2}, and hence is omitted.
\end{proof}

\begin{proof}[Proof of \cref{main4}]
\cref{main4} is an immediate consequence of combining \cref{main5} and \cref{prp uVHS}.
\end{proof}

\newcommand{\etalchar}[1]{$^{#1}$}

\end{document}